\documentclass[reqno]{amsart}
\usepackage{graphicx}
\usepackage{geometry}
\usepackage[leqno]{amsmath}
\usepackage{amscd}
\usepackage{amsthm}
\usepackage{amsfonts}
\usepackage{amssymb}
\usepackage[colorlinks=true]{hyperref}
\usepackage{xcolor}
\usepackage{mathrsfs}
\newtheorem{theorem}{Theorem}[section]

\newtheorem{corollary}[theorem]{Corollary}
\newtheorem{proposition}[theorem]{Proposition}
\newtheorem{lemma}[theorem]{Lemma}

\theoremstyle{definition}
\newtheorem{definition}[theorem]{Definition}
\newtheorem{example}[theorem]{Example}

\newtheorem{remark}[theorem]{Remark}
\makeatletter
\newcommand{\leqnomode}{\tagsleft@true}
\newcommand{\reqnomode}{\tagsleft@false}
\makeatother

\DeclareMathOperator{\End}{End}
\DeclareMathOperator{\Aut}{Aut}
\DeclareMathOperator{\Inn}{Inn}
\DeclareMathOperator{\Ab}{Ab}

\DeclareMathOperator{\Gal}{Gal}
\DeclareMathOperator{\GL}{GL}
\DeclareMathOperator{\SL}{SL}
\DeclareMathOperator{\CC}{CC}

\DeclareMathOperator{\Perm}{Perm}

\newcommand{\gen}[1]{\langle #1 \rangle}

\newcommand{\id}{\mathrm{id}}

\newcommand{\ca}{\circ_{\alpha}}
\newcommand{\cb}{\circ_{\beta}}
\newcommand{\pa}{\psi_{\alpha}}
\newcommand{\pb}{\psi_{\beta}}
\newcommand{\stb}{\star_{\beta}}

\newcommand{\E}{\mathscr{E}}

\numberwithin{equation}{section}

\begin{document}
\newtheorem*{thm}{Theorem}

\title{Commutator-central maps, brace blocks, and {H}opf-{G}alois structures on {G}alois extensions}
\author{Alan Koch}
\date{\today       }

\begin{abstract} Let $G$ be a nonabelian group. We show how a collection of compatible endomorphisms $\psi_i:G\to G$ such that $\psi_i([G,G])\le Z(G)$ for all $i$ allows us to construct a family of bi-skew braces called a brace block. We relate this construction to other brace block constructions and interpret our results in terms of Hopf-Galois structures on Galois extensions. We give special consideration to the case where $G$ is of nilpotency class two, and we provide several examples, including finding the maximal brace block containing the group of quaternions.
	\end{abstract}


\maketitle

\section{Introduction}

Skew left braces were introduced by Guarnieri and Vendramin in \cite{GuarnieriVendramin17} to construct non-degenerate, set-theoretic solutions to the Yang-Baxter equation, generalizing the work of Rump \cite{Rump07}, who required his solutions to be involutive. It was shown in \cite{SmoktunowiczVendramin18} that skew left braces could be constructed from subgroups $N\le \Perm(G)$, the permutations of a finite group $G$, which are regular and stabilized under conjugation by $\lambda(g)$, where $\lambda:G\to G$ is left regular representation; we refer to such $N$ as \textit{$G$-stable} subgroups. 

For $G$ a finite nonabelian group, in \cite{Koch21} we show how any $\psi\in\End(G)$ such that $\psi(G)$ is abelian gives rise to a regular, $G$-stable subgroup $N$ of $\Perm(G)$, hence a skew left brace. In fact, one does not need to construct $N$ in order to realize the skew left brace: it can be derived directly from $\psi$, and holds even for $G$ infinite. We then use this $\psi$ to construct a potentially large family of skew braces called a \textit{brace block} in \cite{Koch22}. Since then, these constructions have been generalized in various works, e.g. \cite{BardakovNeshchadimYadav22,CarantiStefanello21,CarantiStefanello22,StefanelloTrappeniers22}. In particular, \cite{CarantiStefanello21} relaxes the condition on $\psi\in\End(G)$: if $\psi(ghg^{-1}h^{-1})\in Z(G)$ for all $g,h\in G$ then the construction of skew braces and skew brace blocks found in \cite{Koch21,Koch22} are still valid. We call such an endomorphism a \textit{commutator-central map}. Furthermore, in the unpublished work \cite{CarantiStefanello21a} the authors expand the brace blocks found in \cite{Koch22} to obtain larger families.

The aim of this work is to present a further generalization to \cite{CarantiStefanello21a} in two meaningful ways. The skew left braces in the brace blocks constructed in \cite{CarantiStefanello21a} are characterized by polynomials in $\psi\mathbb Z[\psi]$: we will characterize ours by picking $\psi\alpha$ (function composition) where $\alpha:G\to G$ comes from a significantly larger set. Additionally, we will show how to construct brace blocks that arise from two different commutator-central maps whose images commute with each other modulo the center.

In the time since first developing this manuscript, the works of \cite{BardakovNeshchadimYadav22,CarantiStefanello21,StefanelloTrappeniers22} have emerged. The results of the work here can be recast using the theories presented in these papers. We explain the connections between our work and these other theories (which themselves are interconnected). The constructions presented here can be viewed as a nice, concrete set of examples for the other works.

The paper is organized as follows. After introducing necessary concepts and, importantly, notation, we move on to our theory of brace blocks arising from a commutator-central map on a nonabelian group. We then make the connection to all of the other works mentioned above. As before, while the initial theory of braces arising from abelian maps (which built upon the theory of fixed-point free abelian maps in \cite{Childs13}, itself built upon \cite{ByottChilds11}, and \cite{KochStordyTruman20}) arose from the classification of regular, $G$-stable subgroups of $\Perm(G)$, one can study the construction without reference to these subgroups; that is the approach we take here. However, the regular, $G$-stable subgroups can be easily constructed, and we show how they give rise to Hopf-Galois structures on Galois extensions. We then return to the brace blocks, and consider the case where $G$ has nilpotency class two; we show how the theory greatly simplifies and how one can realize ``opposite'' braces (in the sense of \cite{KochTruman20}) in a brace block. 

Finally, we give several examples. We construct a maximal brace block on the quaternion group, and use the symmetric group to illustrate the necessity of working with commuting commutator-central maps. We show how the brace block constructed in \cite[Ex. 6.4]{Koch22}, starting with the nonabelian group of order $pq$, with $p,q$ prime and $p\equiv 1\pmod q$, can be greatly simplified in this new theory. The original motivation for the study of abelian maps was to study Hopf-Galois extensions on finite field extensions; we conclude this paper with an example which constructs a brace block of infinite size.

Throughout, $G$ is assumed to be a nonabelian group except where specified in Section \ref{Hopf}. We denote its center by $Z(G)$ and its commutator subgroup by $[G,G]$. Furthermore, for $g\in G$ we let $C(g)\in\Aut(G)$ denote conjugation by $g$. For $n\in\mathbb Z^{+}$ we let $C_n$ denote the cyclic group of order $n$.

\section{Skew Left Braces, Bi-Skew Braces, and Brace Blocks}

Here we provide the necessary definitions and tools for the main results. 
We start by reminding the reader of the concept of a skew left brace. While the concept is well-known, we provide a definition here in order to set notation.

\begin{definition}
	A \textit{skew left brace} is a triple $(B,\cdot,\circ)$ where $B$ is a set, $(B,\cdot)$ and $(B,\circ)$ are groups, and the following \textit{brace relation} holds:
	\[a\circ(b\cdot c) = (a\circ b)\cdot a^{-1}\cdot (a\circ c),\;a,b,c\in B,\]
	where $a^{-1}$ is the inverse to $a\in (G,\cdot)$.
\end{definition}

Simple examples include $(G,\cdot,\cdot)$ where $(G,\cdot)$ is a group; and $(G,\cdot,\cdot')$ where $g\cdot'h = h\cdot g$. We call these the \textit{trivial brace on $G$} and the \textit{almost trivial brace} on $G$ respectively. 

Throughout, we will refer to a skew left brace as simply a \textit{brace}. Historically, the construction above is a generalization of a ``left brace'' which requires $(G,\cdot)$ to be abelian. Furthermore, when no confusion arises we will write $a\cdot b$ as $ab$.

It is well-known that if $(B,\cdot,\circ)$ is a brace then the two groups $(B,\cdot)$ and $(B,\circ)$ share an identity which we denote $1_B$. We will denote the inverse to $a\in(B,\circ)$ by $\overline a$.

Notice that the brace relation is not symmetric: in general we do not have $a(b\circ c)=(ab)\circ \overline a \circ (ac)$. However, if this relation also occurs we see that both $(B,\cdot,\circ)$ and $(B,\circ,\cdot)$ are braces, a structure first considered in \cite{Childs19} and called a \textit{bi-skew brace}.

Braces were developed to provide set-theoretic solutions to the Yang-Baxter equation. A \textit{set-theoretic solution to the Yang-Baxter equation} is a set $B$ and a map $R:B\times B\to B\times B$ such that $(B\times\id)(\id \times B)(B\times \id)=(\id \times B)(B\times \id)(\id \times B):B^3\to B^3$. Braces give non-degenerate set-theoretic solutions to the YBE: if $(B,\cdot,\circ)$ is a brace then 
\[R(a,b)=(a^{-1}(a\circ b), \overline{a^{-1}(a\circ b)}\circ a \circ b)\]
is a solution, and its inverse (see, e.g., \cite[Th. 4.1]{KochTruman20}) is
\[R'(a,b)=((a\circ b)a^{-1}, \overline{(a\circ b)a^{-1}}\circ a \circ b).\]

Our focus here will be in constructing families of interconnected braces, first considered in \cite{Koch22}.

\begin{definition}
	A \textit{brace block} consists of a set $B$ together with a collection of binary operations $\{\circ_i:i\in\mathscr I\}$ on $B$ such that $(B,\circ_i,\circ_j)$ is a brace for each $i,j\in\mathscr I$.
\end{definition}

We will denote a brace block by $(B,\mathscr O)$, where $\mathscr O=\{\circ_i:i\in\mathscr I\}$. Note that every brace in a brace block is necessarily bi-skew.

We will now introduce a natural notion of substructures.

\begin{definition}
	Let $(B,\mathscr O)$ be a brace block. Let $A\subseteq B$ and $\mathscr O'\subseteq \mathscr O$ be chosen such that $A$ is a subgroup of $(B,\circ)$ for all $\circ\in\mathscr O'$. Then $(A,\mathscr O')$ is a \textit{sub-brace block} of $(B,\mathscr O)$ and we write $(A,\mathscr O')\subseteq (B,\mathscr O)$. Alternatively, we say $(B,\mathscr O)$ \textit{contains} $(A,\mathscr O')$.
	
	A sub-brace block $(A,\mathscr O')\subseteq (B,\mathscr O)$ is \textit{proper} if either $A\ne B$ or $\mathscr O'\ne\mathscr O$, and we say that $(B,\mathscr O)$ is \textit{maximal} if it is not contained in another proper brace block.
\end{definition} 

If we let $A=B$ then any $\mathscr O'\subseteq \mathscr O$ produces a sub-brace block. This is, in fact, the only type of sub-brace block of interest to us here. We have introduced the more general definition above for use in future work.

\begin{example}
	The trivial brace $(G,\cdot,\cdot)$ on a group $G$ gives the brace block $(G,\{\cdot\})$. Similarly, the almost trivial brace on a nonabelian group $G$ gives us the brace block $(G,\{\cdot,\cdot'\})$.
\end{example}

\begin{example}
	More generally, if $(B,\cdot,\circ)$ is any bi-skew brace, then $(B,\{\cdot,\circ\})$ is a brace block.
\end{example}

We will see examples of much larger brace blocks later.

\section{The Main Construction}

In this section, we will show how to construct large families of brace blocks from a single commutator central map, encompassing and extending the results in \cite{Koch22} and \cite{CarantiStefanello21a}. 

We say $\psi\in\End(G)$ is \textit{commutator-central} if $\psi([G,G])\le Z(G)$. Equivalently, $\psi$ is commutator-central if and only if for $g,h\in G$ there exists a $z\in Z(G)$ such that $\psi(gh)=\psi(hg)z$; or simply $\psi(gh)\equiv \psi(hg)\pmod{Z(G)}$. Of course, any abelian map is commutator-central, taking $z=1_G$ above. We denote the set of commutator-central maps by $\CC(G)$ and the set of abelian maps by $\Ab(G)$.

Let $\E$ denote the free group generated by $\End(G)$. Then $\E$ acts on $G$ as follows: for $\alpha=n_1\phi_1+\cdots + n_t\phi_t,\;\phi_1,\dots,\phi_t\in \End(G)$, 
\[\alpha(g) = \phi_1(g^{n_1})\phi_2(g^{n_2})\cdots \phi_t(g^{n_t}).\]
Any $n\in\mathbb Z$ gives rise to an element of $\E$, namely $n\cdot\mathrm{id}_G$, which we also denote $n$. In particular, $0\in\E$ is the trivial map and $1\in\E$ is the identity.

Note that we can express any $\alpha\in\E$ as a sum/difference of endomorphisms (without integer coefficients), where $\phi_1-\phi_2$ is defined in the obvious way.

We shall be interested in compositions $\psi\alpha$, where $\psi\in\CC(G)$ and $\alpha\in\E$. For compactness, we will denote $\psi\alpha$ by $\pa$.

We start with the following basic properties.
\begin{lemma}\label{theLemma} For $\psi\in\CC(G),\;\alpha\in \E$ we have:
	\begin{enumerate}
		\item\label{a} $\pa(1_G)=1_G$;
		\item\label{d} $\psi_{\alpha+\beta}\equiv\psi_{\beta+\alpha}\pmod {Z(G)}$;
		\item\label{b} $\pa(gh)\equiv \pa(g)\pa(h) \pmod {Z(G)}$;
		\item\label{e} $\pa(g)^{-1}\equiv \pa(g^{-1})\pmod {Z(G)}$;
		\item\label{f} $\pa(gh)\pa(g)^{-1}\equiv \pa(h)\pmod {Z(G)}$.
	\end{enumerate}
\end{lemma}

\begin{proof}
	That (\ref{a}) holds is obvious since $\alpha(1_G)=1_G$. Also, for all $\alpha,\beta\in \E$ we have
	\[\psi(\alpha(g)\beta(g))\psi(\alpha(g)^{-1}\beta(g)^{-1})=\psi(\alpha(g)\beta(g)\alpha(g)^{-1}\beta(g)^{-1})\in Z(G),\] thus  $\psi(\alpha(g)\beta(g))\equiv\psi(\beta(g)\alpha(g)\pmod {Z(G)}$ and (\ref{d}) holds.

	If $\alpha$ is an endomorphism then so is $\pa$ and (\ref{b}) is clear. Now suppose (\ref{b}) holds from some $\alpha\in \E$ and let $\beta=\alpha+\phi,\;\phi\in\End(G)$. Then
	\begin{align*}
		\pb(gh)&=\psi(\alpha(gh)\phi(gh))\\
		&=\psi(\alpha(gh)\phi(g)\phi(h))\\
		&=\psi(\alpha(gh))\psi(\phi(g)\phi(h)) \\
		&\equiv \psi(\alpha(g))\psi(\alpha(h)) \psi(\phi(g))\psi(\phi(h)) \pmod{Z(G)}\\
		&\equiv \psi(\alpha(g)\alpha(h)\phi(g)\phi(h))\pmod {Z(G)}\\
		&=\pa(g)\psi(\alpha(h)\phi(g))\psi(\phi(h))\\
		&\equiv\pa(g)\psi(\phi(g)\alpha(h))\psi(\phi(h)) \pmod {Z(G)}\\
		&=\pb(g)\pb(h).
	\end{align*}
Thus, by induction, (\ref{b}) holds. (\ref{e}) follows by setting $h=g^{-1}$. Finally, (\ref{f}) holds by (\ref{b}) and the fact that $\psi\in\CC(G)$.

\end{proof}

We are now in a position to give our brace construction. For $\psi\in\CC(G),\;\alpha\in\E$, define a binary operation $\circ_{\psi,\alpha}$ on $G$ by \[ g\circ_{\psi,\alpha} h =g\pa(g)h\pa(g)^{-1}.\]

When $\psi$ is understood we will suppress the notation and write $g\ca h$. Evidently, if $x \equiv \pa(g) \pmod {Z(G)}$ then $g\ca h = gxhx^{-1}$. This will allow us to use the results in Lemma \ref{theLemma} to simplify computations in the proofs that follow.

\begin{proposition}
	Let $\psi\in\CC(G),\;\alpha\in \E$. Then $(G,\ca)$ is a group.
\end{proposition}

\begin{proof}
	That $1_G$ is the identity is a simple computation, as is that the inverse to $g\in(G,\ca)$ is $\pa(g)^{-1}g^{-1}\pa(g)$, To show associativity we have, by Lemma \ref{theLemma} (\ref{b}),
	\begin{align*}
		(g\ca h)\ca k &= (g\pa(g)h\pa(g)^{-1})\ca k\\
		&=(g\pa(g)h\pa(g)^{-1})\pa(g)\pa(h)k\pa(h)^{-1}\pa(g)^{-1}\\
		&=g\pa(g)h\pa(h)k\pa(h)^{-1}\pa(g)^{-1}\\
		&=g\pa(g)(h\ca k)\pa(g)^{-1}\\
		&=g\ca(h\ca k).
	\end{align*}
\end{proof}

\begin{example}\label{AbAK}
	Let $\psi\in\Ab(G)$, and let $\alpha = -1$. Then
	\[g\ca h = g\pa(g)^{-1}h\pa(g)=g\psi(g^{-1})h\psi(g),\]
	obtaining the circle group obtained in \cite{Koch21} for $\psi\in\Ab(G)$. As pointed out in \cite[Th. 1.1]{CarantiStefanello21}, this construction holds for general $\psi\in\CC(G)$.
\end{example}

\begin{example}\label{AbCS}
	Let $\psi\in\Ab(G)$, and let $\alpha = 1$. Then
	\[g\ca h = g\pa(g)h\pa(g)^{-1}=g\psi(g^{-1})h\psi(g),\]
	obtaining \cite[Th. 1.2]{CarantiStefanello21a} in the case $\psi\in\CC(G)$.
\end{example}

We now arrive at the main result.

\begin{theorem}\label{bigTheorem}
	Let $\psi,\psi'\in\CC(G)$ satisfy $[\psi(G),\psi'(G)]\subset Z(G)$. For $\alpha,\beta\in \E$ let  
	\[g\ca h = g\pa(g)h\pa(g)^{-1},\;g\stb h = g\pb'(g) h\pb(g)^{-1},\;g,h\in G.\]
	Then $(G,\ca,\stb)$ is a brace.
\end{theorem}

\begin{proof}
	We merely need to check that the brace relation holds. Write $\widetilde g$ for the inverse to $g\in(G,\ca)$. Since $\widetilde g=\pa(g)^{-1}g^{-1}\pa(g)$ we know $\pa(\widetilde g)\equiv \pa(g^{-1}) \pmod {Z(G)}$ and we get
	\begin{align*}
		(g\stb h)\ca \widetilde g \ca (g\stb k) &= (g\stb h) \ca ((\pa(g)^{-1}g^{-1}\pa(g))\pa(g^{-1})(g\pb'(g)k\pb'(g)^{-1})\pa(g^{-1})^{-1})\\
		&=(g\pb'(g)h\pb'(g)^{-1}) \ca (\pa(g)^{-1}\pb'(g)k\pb'(g)^{-1}\pa(g^{-1})^{-1})\\
		&=g\pb'(g)h\pb'(g)^{-1}(\pa(h)\pa(g))(\pa(g)^{-1}\pb'(g)k\pb'(g)^{-1}\pa(g^{-1})^{-1})(\pa(g)^{-1}\pa(h)^{-1})\\
		&= g\pb'(g)h\pa(h)k\pa(h)^{-1}\pb'(g)^{-1}\\
		&=g\cb(h\ca k).
	\end{align*}
\end{proof}

From this we immediately obtain a brace block.
\begin{corollary}\label{bigCor}
	Let $\psi_1,\psi_2,\dots \psi_t\in\CC(G)$ have the property that $[\psi_i(G),\psi_j(G)]\subset Z(G)$. Then
	\[(G,\bigcup\limits_{i=1}^t\{\circ_{\psi_i,\alpha}:\alpha\in\E\})\]
	is a brace block.
\end{corollary}

We can then apply Corollary \ref{bigCor} to obtain solutions to the Yang-Baxter equation.

\begin{corollary}
	Let $\psi_1,\psi_2,\dots, \psi_t\in\CC(G)$ have the property that $[\psi_i(G),\psi_j(G)]\subset Z(G)$ for all $1\le i,j\le t$. Then for all $1\le i,j\le t,\;\alpha,\beta\in \E$,
	\begin{align*}
		R(g,h) &= \left(\widetilde{g}\circ_{\psi_i,\alpha}(g\circ_{\psi_j,\beta} h),\overline{\widetilde{g}\circ_{\psi_i,\alpha}(g\circ_{\psi_j,\beta} h)}\circ_{\psi_j,\beta}g\circ_{\psi_j,\beta}h\right)\\
			R'(g,h) &= \left((g\circ_{\psi_j,\beta}h)\circ_{\psi_i,\alpha}\widetilde{g} ,\overline{(g\circ_{\psi_j,\beta} h)\circ_{\psi_i,\alpha}\widetilde{g}}\circ_{\psi_j,\beta}g\circ_{\psi_j,\beta}h\right)
	\end{align*}
are inverse non-degenerate, set-theoretic solutions to the Yang-Baxter equation, where $g\circ_{\psi_i,\alpha} \widetilde g = g\circ_{\psi_j,\beta}\overline g = 1_G$. Furthermore, $R$ is involutive (and hence $R'=R$) if and only if $(G,\circ_{\psi_i,\alpha})$ is abelian.
\end{corollary} 

Note that it is possible--but messy--to write explicit formulas for $R$ and $R'$ which does not directly use the group operations in the style of \cite[Th. 5.1]{KochStordyTruman20}.

\section{Connections with Previous and Concurrent Works}

In this section, we connect the work presented here with that in the literature. We have already seen through Examples \ref{AbAK} and \ref{AbCS} how the braces constructed here are generalizations of the abelian map constructions. The following shows that we recover the corresponding brace blocks as well.

\begin{example}
	Let $\psi\in\Ab(G)$, and for each $n\ge 0$ let \[\alpha_n = \sum_{i=0}^{n-1} (-1)^{i}\binom n i \psi^i.\] Then $\psi_{\alpha_n} = -(1-\psi)^n+1$, and writing $g\circ_n h$ for $g\circ_{\alpha_n} h$ we obtain the brace block $(G,\{\circ_n:n\ge 0\})$ constructed in \cite{Koch22}.
\end{example}

Note that the results from \cite{CarantiStefanello21a} can also be obtained, though their recursive description of the operation makes it somewhat messy to do so. 

In \cite{CarantiStefanello22} a more general construction is given as follows. (Note that we are changing the notation to better align with the work here.) Let $G_1,G_2\le G$ be chosen such that $G_2\le Z(G)$ and $G_1/G_2:=G_1/G_1\cap G_2$ is abelian. Let $\gamma:G/G_2\to G_1/G_2$ be a homomorphism, and let $\gamma^{\uparrow}:G\to G_1$ be any lifting of $\gamma$, not necessarily a homomorphism. Define \[g\circ h = g\gamma^{\uparrow}(g)h \gamma^{\uparrow}(g)^{-1}.\]  
This operation is independent of choice of lift. Then $(G,\cdot,\circ)$ is a bi-skew brace. One can extend this further by applying a bilinear map to the right of the expression for $g\circ h$ above, but that will not be needed here. By varying $\gamma$ (but keeping $G_1,G_2$ fixed) one obtains a brace block.

Our construction, developed independently, can be seen as a special case of \cite{CarantiStefanello22}. Indeed, let $\psi\in\CC(G)$, and set $G_1=\psi(G)$ and $G_2=Z(G)$. For $\alpha\in \E$ we have $\pa:G\to \psi(G)$ is a lift of a homomorphism $G/Z(G)\to\psi(G)/Z(G)$. More generally, given $\psi_1,\psi_2,\dots, \psi_t\in\CC(G)$ one can take $G_1=\psi_1(G)\psi_2(G)\cdots\psi_t(G)\le G$, keep $G_2=Z(G)$, and recover the full brace block from Corollary \ref{bigCor}. 

One can also observe that $C(\pa(g)):G\to \Aut(G)$ is an example of a ``$\lambda$-homomorphism'' and can be used in the construction of brace blocks as found in \cite{BardakovNeshchadimYadav22}. In that work the authors use the term ``symmetric brace system'' for ``brace block".

Finally, one also can see how our construction is a special case of the classification found in \cite[Th. 6.6]{StefanelloTrappeniers22}, also developed independently of our theory. In that work, adapting notation as needed, braces are constructed by considering homomorphisms $\gamma:G\to M\le\Aut(G),\; g\mapsto \gamma_g$ with $M$ abelian such that $\gamma_{\chi(g)}=\gamma_g$ for all $g\in G,\; \chi\in\gamma(G)$. Given such a choice define 
\[g\circ_{\gamma} h = g \gamma_g(h). \]
Then $(G,\cdot,\circ_{\gamma})$ is a brace; moreover, $(G,\circ_{\gamma_1},\circ_{\gamma_2})$ is a brace for all choices of $\gamma_1,\gamma_2:G\to\Aut(G)$ satisfying the conditions above.

We now relate the Stefanello-Trappeniers construction to the work presented here; note that this connection is essentially an application of \cite[Ex. 6.12]{StefanelloTrappeniers22}. Pick $\psi\in\CC(G)$ and $\alpha\in \E$. Let $\gamma:G\to\Inn(G)$ be given by
\[\gamma_g(h)=\pa(g) h \pa(g)^{-1}.\] 
That is, $\gamma_g=C(\pa(g))$. Since $\psi\in\CC(G)$ it follows that the image of $\gamma$ is an abelian subgroup of $\Inn(G)$, and clearly $g\ca h = g\circ_{\gamma} h$. 

While the theory here encompasses the work in \cite{CarantiStefanello21a} it seems certain that the constructions in \cite{CarantiStefanello22} and \cite{StefanelloTrappeniers22} can be used to construct larger brace blocks than our work above. However, our theory should be easier to implement in general, allowing for the quick creation of brace blocks. For example, \cite{CarantiStefanello21a} requires selecting $G_1,G_2$ in advance of picking the map, whereas in our theory picking the map determines $G_1$ (and $G_2$ is predetermined as well).

\section{Brace Blocks and Hopf-Galois Structures}\label{Hopf}

In this section we show how brace blocks give rise to Hopf-Galois structures on Galois field extensions. To do so, we need to recall the connection between Hopf-Galois structures and certain groups of permutations. Throughout this section, $G$ will denote a \textit{finite} group, initially allowed to be abelian. 

Let $\Perm(G)$ denote the group of permutations of $G$; we denote the image of $g\in G$ under $\eta\in\Perm(G)$ by $\eta[g]$. We say a subgroup $N$ of $\Perm(G)$ is \textit{regular} if for all $g,h\in G$ there exists a unique $\eta\in N$ such that $\eta[g]=h$. Alternatively, $N\le\Perm(G)$ is regular if and only if:
\begin{enumerate}
	\item for all $\eta\ne 1_N$, $\eta[g]\ne g$ for all $g\in G$;
	\item for all $g,h\in G$ there exists an $\eta\in N$ such that $\eta[g]=h$;
\end{enumerate}
If $G$ is finite, we can replace either condition with insisting that $|N|=|G|$.

Two examples of regular subgroups of $\Perm(G)$ are $\lambda(G)$, the group of left regular representations, and $\rho(G)$, the group of right regular representations. 

Of particular interest to us are the regular subgroups of $\Perm(G)$ which are $G$-stable. Here, we say $N\le \Perm(G)$ is \textit{$G$-stable} if, for all $g\in G$, conjugation by $\lambda(g)$ restricts to an automorphism of $N$. This property is also referred to as ``normalized by $G$'' in the literature.
 For $\eta\le\Perm(G)$ we denote $\lambda(g)\eta\lambda(g^{-1})$ by $^g\eta$.

For example, $\lambda(G)$ is obviously $G$-stable, as is $\rho(G)$ since $^g\rho(g)=\lambda(g)\rho(g)\lambda(g^{-1}) = \rho(g)$ for all $g,h\in G$.

Regular, $G$-stable subgroups allow us to find Hopf-Galois structures on $G$-Galois extensions \cite{GreitherPareigis87}: if $L/K$ is Galois with $G=\Gal(L/K)$ and $N\le\Perm(G)$ is regular, $G$-stable then $G$ acts on the group algebra $L[N]$ on $L$ by the Galois action and on $N$ by conjugation by $\lambda(G)$. Let $H=L[N]^G$ be the set of fixed points under this action. Then $H$ is a $K$-Hopf algebra which acts on $L$ providing a Hopf-Galois structure on $L/K$. This Hopf-Galois structure is said to be of \textit{type} $N$, where by abuse of notation $N$ is the abstract group isomorphic to the regular, $G$-stable subgroup above.

For the remainder of this section, we once again assume $G$ is nonabelian (but still finite). Suppose $L/K$ is a finite Galois field extension with $\Gal(L/K)=G$. Let $\psi\in\CC(G),\;\beta\in\E$. For each $g\in G$, let $\eta_g\in\Perm(G)$ be given by \[\eta_g[h]=g\cb h = g\pb(g)h\pb(g)^{-1},\;h\in G.\]

It is useful to observe that we can write $\eta_g=\lambda(g\pb(g))\rho(\pb(g))$ or $\eta_g=\lambda(g)C(\pb(g))$. Let $N=\{\eta_g:g\in G\}$. Since $\eta_g[1_G]=g$ for all $g\in G$ we see that $|N|=|G|$; furthermore, $\eta_g[h]=\eta_g[k]$ implies $C(\pb(g))[h]=C(\pb(g))[k]$ from which it follows that $h=k$. Thus, $N$ is regular. It is also $G$-stable since $^k\eta_g = \eta_{kg\pb(g)k^{-1}\pb(g)^{-1}}$ as can be checked. As $N\le\Perm(G)$ is regular, $G$ stable the Hopf algebra $H=L[N]^G$ acts on $L$, making $L/K$ an $H$-Hopf Galois extension. Note that the argument presented here is a straightworawd generalization of the argument in \cite{Koch21}.

Thus, each $\beta\in\E$ provides a Hopf-Galois structure on $L/K$, and two choices $\beta,\beta'\in\E$ give the same structure if and only if $\psi(\beta-\beta')(G)\subseteq Z(G)$. 

In general, it can be difficult to uncover many of the properties of $H$, however the group-like elements of $H$ can be identified. As is well-known, the group-likes in $H$ are the $\eta\in N$ upon which $G$ acts trivially. In our setting, if $^k\eta_g = \eta_g$ then $kg\pa(g)k^{-1}\pa(g)^{-1} = g$ and the group-likes of $H$ can be seen to correspond to $g\in G$ such that $g\pb(g)\in Z(G)$.

In \cite{SmoktunowiczVendramin18} a slightly different connection occurs between Hopf-Galois structures and braces, a generalization of the work found in \cite{Bachiller16} (who worked with non-skew braces). Given a $G$-Galois extension of type $N$, one typically constructs a brace $(B,\cdot,\circ)$ with $(B,\cdot)\cong N$ and $(B,\circ)\cong G$. With the theory presented here, the constructed brace $(B,\cdot,\circ)$ has $(B,\cdot)\cong G$ and $(B,\circ)\cong N$. We have chosen this latter interpretation, which has been used implicitly in \cite{KochStordyTruman20,Koch21,Koch22} for simplicity: the ``fixed group'' in our discussion is $G$, whose operation is expressed multiplicatively. If one starts with a Hopf-Galois structure constructed using the holomorph of a fixed type $N$ (see \cite{Byott96}) then the perspective in \cite{SmoktunowiczVendramin18} would provide for better notation.

Of course, our brace blocks contain braces other than those of the form $(G,\cdot,\cb)$; more generally, $(G,\circ_{\psi_i,\alpha},\circ_{\psi_j,\beta})$ is a brace (following the notation in Theorem \ref{bigTheorem}). This leads to regular, stable subgroups involving other groups. Clearly, $\Perm(G)=\Perm(G,\circ_{\psi_j,\beta})$ so the $N$ constructed above is a regular subgroup of $\Perm(G,\circ_{\psi_j,\beta})$ isomorphic to ($G,\circ_{\psi_i,\alpha}$). It is also normalized by conjugation in $(G,\circ_{\psi_j,\beta})$: if $g \circ_{\psi_i,\alpha} \widetilde g = g\circ_{\psi_j,\beta}\overline g = 1_G$ then 
\[^k\eta_g = \eta_{k\circ_{\psi_j,\beta} (g\circ_{\psi_i,\alpha}\widetilde k)}=\eta_{(k\circ_{\psi_j,\beta} g)\circ_{\psi_i,\alpha}\overline k}=\eta_{k\psi_{j,\beta}(k)g\psi_{j,\beta}(k)^{-1}\psi_{i,\alpha}(g)\psi_{i,\alpha}(k)^{-1}k^{-1}\psi_{i,\alpha}(k)\psi_{i,\alpha}(g)^{-1}}.\]
Thus $H=L[(G,\circ_{\psi_i,\alpha})]^{(G,\circ_{\psi_j,\beta})}$ gives a Hopf-Galois structure on any $(G,\circ_{\psi_j,\beta})$ Galois extension of type $(G,\circ_{\psi_i,\alpha})$. 

We can generalize this to any brace block with finite underlying set.
\begin{proposition}
	Let $(B,\mathscr O)$ be a brace block with $B$ finite. Then there exist Hopf-Galois structures on a $(G,\circ)$-Galois extension of type $(G,\star)$ for all $\circ,\star\in \mathscr O$.
\end{proposition}

Finally, we observe that if we consider the sub-brace block $(G,\mathscr O')$ where $(G,\circ)\cong G$ for all $\circ\in\mathscr O'$ then there is a strong relationship between $(G,\mathscr O')$ and normalizing graphs as described by Kohl in \cite{Kohl22}.

\section{Brace Blocks on Groups of Nilpotency Class Two and Opposite Braces}

In general, the brace block obtained from $G$ depends on the choice of $\psi\in\CC(G)$. However, this is greatly simplified in the case where $G$ is of nilpotency class exactly two. Recall that $G$ is of nilpotency class one if $G$ is abelian, and of nilpotency class two if $[G,G]\le Z(G)$. As we have assumed throughout that $G$ is nonabelian we use ``nilpotency class two'' to mean nilpotency class exactly two.

Suppose $[G,G]\le Z(G)$. Then the identity map is commutator-central, so we may take $\psi=1$. We then have
\[g\ca h = g\alpha(g)h\alpha(g)^{-1},\;g,h\in G\] 
for all $\alpha\in\E$. It is important to notice that, while there are other choices for commutator-central maps there is no compelling reason to use $\psi\ne 1$. Indeed, suppose $\psi\in\CC(G),\psi\ne 1$. We then obtain the brace block $(G,\{\circ_{\psi,\beta} : \beta\in \E\})$. Then $\psi\beta\in \E$ as well, so 
\[g \circ_{\psi,\beta} h = g\circ_{1,\psi\beta}.\]
In other words, $(G,\{\circ_{\psi,\beta}:\beta\in\E\})\subseteq (G,\{\circ_{1,\alpha}:\alpha \in \E\})$ for all $\psi\in\CC(G)$.

Another interesting aspect of the theory in nilpotency class two can be found when considering the theory of opposite braces. Developed independently in \cite{KochTruman20} and \cite{Rump19}, the opposite of a brace provides the inverse to the given solution to the Yang-Baxter equation and provides the opposite Hopf-Galois structure.

Given a brace $(B,\cdot,\circ)$, the \textit{opposite} brace is $(B,\cdot',\circ)$ where $a\cdot' b = ba$. In other words, $(B,\cdot')$ is the opposite group to $(B,\cdot)$. It is not hard to show that the map $a\mapsto a^{-1}$ is a brace isomorphism $(B,\cdot',\circ)\to(B,\cdot,\widehat{\circ})$ where
\[a\widehat{\circ} b = (a^{-1}\circ b^{-1})^{-1},\]
and we can think of $(B,\cdot,\widehat{\circ})$ as the opposite brace to $(B,\cdot,\circ)$ as well. Note that the corresponding solution to the YBE is not inverse to the original brace but becomes so upon applying $a\mapsto a'$ again.

Return to $G$ of nilpotency class two and $\psi=1$, and let $\alpha\in \E$. Then
\[g\widehat{\circ}_{\alpha} h = (g^{-1}\alpha(g^{-1})h^{-1}\alpha(g^{-1})^{-1})^{-1} = \alpha(g^{-1})h\alpha(g^{-1})^{-1}g\]
Let $\alpha = \phi_1+\phi_2+\cdots+\phi_t$ we define $\alpha^*=\phi_t+\phi_{t-1}+\cdots+\phi_1$. In light of Lemma \ref{theLemma} (\ref{d}) we have $g\ca h = g\circ_{\alpha^*} h$, of course; furthermore $\alpha(g^{-1})=\alpha^*(g)^{-1}$.  Let $\beta=-1-\alpha^*\in \E$. Then $\beta(g)=g^{-1}\alpha^*(g)^{-1} = g^{-1}\alpha(g^{-1})$ and we get
\begin{align*}
	g\cb h &= g\beta(g)h\beta(g)^{-1}\\
	&=g(g^{-1}\alpha(g^{-1}))h(g^{-1}\alpha(g^{-1}))^{-1}) \\
	&=\alpha(g^{-1})h\alpha(g^{-1})^{-1}\\
	&=g\widehat{\circ}_{\alpha} h.
\end{align*}

We thus obtain the following.
\begin{proposition}
	Let $G$ be of nilpotency class two, and let $\psi=1$. Then $(G,\cdot,\ca)$ is in a brace block if and only if $(G,\cdot,\widehat{\circ}_{\alpha})$ is.
\end{proposition}

\section{Examples}

We conclude with four examples that exhibit very different behavior.

\begin{example}[The group $Q_8$]\label{Q8}
	Let $G=Q_8$ be the group of quaternions. Write $G=\gen{a,b:a^4=b^4=a^2b^2=abab^{-1}=1_G}$. Since $[G,G]=Z(G)=\gen{a^2}$ we may take $\psi=1\in\CC(G)$ as above. Define $\psi_1,\psi_2,\psi_3,\psi_4\in\End(G)$ by
	\[\begin{array}{c | c c c c}
		\;&\phi_1&\phi_2&\phi_3&\phi_4\\\hline
		a & a^3b & a^2b & b & a^3\\
		b & a^3 & a^3 & ab & ab
		\end{array},\]
	and let $\sigma\in \Perm(\{a,b,ab\})$. We will find regular, $G$-stable subgroups of $\Perm(G)$ from which the braces can be readily computed. Recall that $\psi\in\CC(G),\;\alpha\in \E$ gives us $N=\{\eta_g:g\in G\}$ with $\eta_g[h]=\lambda(g\pa(g))\rho(g)[h] = g\pa(g)h\pa(g)^{-1}$, hence the circle operation in the corresponding brace is $g\ca h = \eta_g[h]$. 
	
	We will show how certain choices of $\alpha\in \E$ give us different regular subgroups $N\le\Perm(G)$. First, let $\alpha = 0$. Then $\eta_g=\lambda(g\cdot 1_G)\rho(1_G)=\lambda(g)$ hence $N=\lambda(G)$. On the other hand, taking $\alpha=-1$, the map-to-inverse function, gives $\eta_g=\lambda(g\cdot g^{-1})\rho(g^{-1})=\rho(g^{-1})$ and $N=\rho(G)$.
	
	Now let $\alpha=\phi_1$. Then 
	\begin{align*}
		\eta_{a^2}&=\lambda(a^2\cdot a^2)\rho(a^2) = \lambda(a^2)\\
		\eta_{a^3b} &=\lambda((a^3b)(b))\rho(b) = \lambda(a)\rho(b)\\
		\eta_{a^3} &=\lambda(a^3 (ab))\rho(ab) = \lambda(b)\rho(ab),
	\end{align*}
and as these three permutation have order $2$ and commute we get that \[N=\gen{\lambda(a^2),\lambda(a)\rho(b),\lambda(b)\rho(ab)} = \gen{\lambda(a^2)}\times\gen{\lambda(a)\rho(b)}\times\gen{\lambda(b)\rho(ab)}\cong C_2\times C_2\times C_2.\]

	Next, let $\alpha=\phi_2+\phi_3$. Then $\alpha(a)=\phi_2(a)\phi_3(a) = a^2b\cdot b = 1_G$. We get
	\[
		\eta_a = \lambda(a\cdot 1_G)\rho(1_G) = \lambda(a),
	\]
	and since $\alpha(a^2b)=\phi_2(a^2b)\phi_3(a^2b)=a\cdot a^3b = b$ we have
	\[\eta_{a^2b}=\lambda(a^2b\cdot b)\rho(b) = \rho(b)\]
	and $N=\gen{\lambda(a),\rho(b)}\cong C_4\times C_2$.
	
	Picking $\alpha=\phi_4$ gives
	\[\eta_{a^3} = \lambda(a^3\cdot a)\rho(a) = \rho(a),\;\eta_{a^3b} = \lambda(a^3b\cdot a^2b)\rho(a^2b) = \lambda(a)\rho(b)\] and $N=\gen{\rho(a),\lambda(a)\rho(b)}\cong D_4$, the dihedral group of order $8$.
	
	Lastly, let $\alpha=-1-\phi_4$. Then $\alpha(a)=a^{-1}(a^3)^{-1} = 1_G,\;\alpha(ab)=(ab)^{-1}b^{-1} = a^3$ and hence
	\[ \eta_{a}=\lambda(a),\;\eta_{ab}=\lambda(ab\cdot a^3)\rho(a^3) = \lambda(aba)\rho(a) =  \lambda(b)\rho(a).\]
	With this choice of $\alpha$ we have $N=\gen{\lambda(a),\lambda(b)\rho(a)}\cong D_4$.
	
	By permuting the generators via $\sigma$--that is, replacing $\phi_i$ with $\sigma\phi_i\sigma^{-1}$ above--we obtain a total of $16$ different regular, $G$-stable subgroups of $\Perm(G)$. This agrees with the classification of regular, $G$-stable subgroups of $Q_8$ as found in \cite{TaylorTruman19}: our construction here yields every Hopf-Galois structure of non-cyclic type. A description of all of the choices for the elements of $\E$ appears below.
	
	\[\begin{array}{c | c c c c c c c c }
	 & 0 & -1 & \phi_{1,a,b} & \phi_{1,b,a} & \alpha_{a,b} & \alpha_{a,ab} & \alpha_{b,a} & \alpha_{b,ab} \\\hline
	a & 1 & a^3 & a^3b & a^2 b & 1 & 1 & a & a^3b \\
	b& 1 & a^2b & a^3 & ab & b & ab & 1 & 1 \\\hline\hline
	& \alpha_{ab,a} & \alpha_{ab,b} & \phi_{4,a,b} & \phi_{4,b,a}& \phi_{4,ab,b} & \beta_{a,b} & \beta_{b,a} & \beta_{ab,b}\\\hline
	a & a & b & a^3 & a^3b & a^2 b & 1 & b & a^3b\\
	b & a^3 & a^2b & ab & a^2b & a^3 & a & 1 & ab\\
	
\end{array}\]
The extra subscripts indicate the particular choices for $\sigma(a)$ and $\sigma(b)$; additionally we write  $\alpha=\phi_2+\phi_3$ and $\beta = -1-\phi_4$ for brevity. Recall that two choices of elements in $\E$ produce the same brace if and only if they differ by an element of the center; since $x=yz,\;x,y\in Q_8,\;z\in Z(Q_8)$ holds if and only if $x=y$ or $x=y^{-1}$ it is easy to see from inspecting the table above that these all produce distinct braces. Keep in mind that we refer to identical braces, not isomorphic braces: some of the braces constructed here will be isomorphic, as shown in \cite[Ex. 3.6]{KochTruman23}. Note that $\alpha_{s,t}$ and $\beta_{s,t},\;s\ne t\in\{a,b,ab\}$ are not completely determined from the generators shown above, but a full determination is not needed here. 
	
	Furthermore, the brace block constructed above is maximal: there is no choice of $\alpha$ such that $(G,\ca)\cong C_8$. This follows directly from the observation that 
	\[|\eta_g| = |\lambda(g\alpha(g))\rho(g)| \le |\lambda(g\alpha(g))|\cdot|\rho(g)| = \operatorname{lcm}(|\lambda(g\alpha(g))|,|\rho(g)|)\le 4.\]
	Alternatively, by \cite[Th. 5.1]{Byott07} we know that a $C_8$-Galois cyclic extension has no Hopf-Galois structures of type $C_4\times C_2$ or $C_2\times C_2\times C_2$. If there was an $\alpha\in\E$ such that $(G,\ca)\cong C_8$ then the brace block would include a brace $(G,\circ,\star)$ with $(G,\circ)\cong C_8$ and $(G,\star)\cong C_4\times C_2$ and, hence, a Hopf-Galois structure on a $C_8$ extension of type $C_4\times C_2$.
\end{example}

\begin{remark}
	The above example illustrates the necessity of using $\E$ instead of $\End(G)$. For example, we use $\alpha=\phi_2+\phi_4\not\in\End(G)$ to obtain the group $\gen{\lambda(a),\rho(b)}$. Suppose there exists an $\phi\in \End(G)$ which gives $N=\gen{\lambda(a),\rho(b)}$. As usual, write $N=\{\eta_g:g\in G\}$. Then $\eta_a=\lambda(a\phi(a))\rho(\phi(a))=\lambda(a)$ implies that $\phi(a)\in Z(G)$, hence $\phi(a)=1_G$ or $\phi(a)=a^2$. On the other hand, \[\eta_b=\rho(b^{-1})=\rho(a^2b) = \lambda(b\phi(b))\rho(\phi(b)),\]
	and so $\phi(b)=b$ or $\phi(b)=a^2b$. Regardless of the choices made above, $\phi(a^2)=\phi(a)^2=1_G$ and $\phi(b^2)=\phi(b)^2\ne 1_G$ which is a contradiction since $a^2=b^2$. Thus we do need to allow sums of endomorphisms to obtain this brace block.
\end{remark}

In the quaternion example, the brace block is determined by varying $\alpha\in \E$ but not $\psi\in\CC(G)$. The following example relies exclusively on our ability to change $\psi$. 

\begin{example}[The symmetric groups]
	Let $G=S_n$ be the symmetric group of order $n$ with $n\ge 5$. Since $Z(S_n)$ is trivial we have $\CC(S_n)=\Ab(S_n)$. Thus the elements of $\CC(S_n)$ have already been described in \cite[Ex. 3.7]{Koch21}: for $\tau\in S_n,\;\tau^2=1_G$ we have
	\[\psi(\sigma)=\begin{cases} 1 & \sigma\in A_n \\ \tau & \sigma\notin A_n \end{cases}.\]
	Furthermore, the braces constructed here, combined with their opposites, account for all braces $(G,\cdot,\circ)$ with $(G,\cdot)=S_n$ (see the classification given in \cite{CarnahanChilds99}).
	
	Now for all $\alpha\in \E$ we have $\pa(G)\subset\psi(G)=\gen{\tau}$. If we pick $\tau'\ne \tau,\;\tau'^2=1_G$ we have another map, say $\psi'\in\CC(G)$, and since $\psi'(G)=\gen{\tau'}$ it is clear that $\psi'\ne\pa$ for any $\alpha\in \E$. Because of this, we can always pick $\alpha\in \E$ to be $\alpha=1$ (or $\alpha=0$ to obtain the trivial brace when $\tau\ne 1_G$). This would create a brace block with at most two binary operations.
	
	Let us for now change notation and write $g\circ_{\tau} h$ for this operation.
	If $\tau\tau'=\tau'\tau$ then $(G,\circ_\tau,\circ_\tau')$ forms a brace. This allows us to generate brace blocks of a different flavor than the one constructed in Example \ref{Q8}. Let $T$ be a maximal elementary abelian $2$-group inside $S_n$. Then $T\cong C_2^{{\lfloor n/2\rfloor}} $, and \[(G,\{\circ_{\tau}:\tau\in T\})\] forms a brace block. 
	
	As $T$ is not unique we obtain several brace blocks here. For example, in $S_5$ there are $5$ choices for $T$, which can be parameterized by the element in $\{1,2,3,4,5\}$ remaining fixed (so, e.g., $T_4=\{\iota, (12), (35), (12)(35)\}$. This leads to five distinct brace blocks, whose ``intersections'' (suitably defined) are easy to compute.
\end{example}

\begin{example}[Metacyclic, order $pq$]
	Let $M=M_{p,q}=\gen{s,t:s^p=t^q=tst^{-1}s^{-d}=1_M}$, where $p\equiv 1\pmod q,\;d^q\equiv 1\pmod p,\;d\not\equiv 1 \pmod p$. In \cite[Ex, 6.4]{Koch22} we show how a brace block can be constructed using $\psi(s)=1_M,\;\psi(t)=t^{1-j}$ where $j$ is a primitive root mod $q$. The brace block, consisting of $q-1$ binary operations, is formed by taking $\psi_n=-(1-\psi)^n+1$. It turns out that $\psi_n(s)=1_M$ and $\psi_n(t)=t^{1-j^n}$. While the sequence of constructed binary operations follows a pattern, if we forget ordering we see that all of the binary operations have the form
	\[g\circ_n h = g\psi_n(g)h\psi_n(g)^{-1},\;\psi_n(s)=1_M,\;\psi_n(t)=t^n,\; 2\le n\le p.\]
	A more direct way to construct this block is to observe that $\psi_n(s)=1,\;\psi_n(t)=t^n,\;2\le n \le p$ is commutator-central (in fact, abelian), and since $\psi_n(G)\le \gen t$ for all $n$ we see that $\{\psi_n:2\le n \le p\}$ consists of pairwise commuting commutator-central maps, hence 
	\[(G,\{\circ_n:g\circ_n h = g\psi_n(g)h\psi_n(g)^{-1},\;2\le n \le p\})\]
	is a more direct way to express the brace block above.
\end{example}

\begin{example}[Matrix groups]
	Let $F$ be any field, not necessarily finite, and let $G=\GL_n(F)$ be the group of invertible $n\times n$ matrices with entries in $F$. For $1\le t \le n$ define $\psi_t:G\to G$ by letting $\psi_t(A)$ be the elementary matrix obtained by multiplying the $t^{\text{th}}$ row of the identity matrix by $d:=\det A$. This is clearly an endomorphism of $G$, and since $\psi$ preserves determinants and $\det([A,B])=1$ we have $\psi([G,G])=I$ so $\psi\in\Ab(G)$. (It is well known that $[G,G]=\SL_n(F)$ provided $G\ne\GL_2(\mathbb F_2)$, and $[\GL_2(\mathbb F_2),\GL_2(\mathbb F_2)]$ is the unique subgroup of order $3$, but all we need here is that $[G,G]\subseteq \SL_n(F)$.) Setting $\alpha=1$ gives $A\circ_{\psi_t,1} B =  A\psi_t(A)B\psi_t(A)^{-1}=AB'$, where $B=[b_{i,j}],\; B'=[b'_{i,j}]$ and

	\[b'_{i,j} = 
	\begin{cases}
		b_{i,j}d & i=t,\;j\ne t\\
		\frac{b_{i,j}}d & i\ne t,\;j=t\\
		b_{i,j} & \text{ otherwise} 
	\end{cases}.\] 	
	Furthermore, $\psi_t(G)$ is contained in the abelian subgroup $D\le\GL_n(F)$ of diagonal matrices, hence $[\psi_t(G),\psi_u(G)]=1_G\in Z(G)$ for all $1\le t,u \le n$. Thus we may use these $n$ commutator-central (in fact, abelian) maps to form blocks.
	
	Let us specialize to the case $G=\GL_2(\mathbb F_p(\!(x)\!))$ where $\mathbb F_p(\!(x)\!),\;p>2$ is the function field over $\mathbb F_p$ in one indeterminant. Then $\phi: G\to G$ given by $\phi([a_{i,j}])=[a_{i,j}^p]$ is an endomorphism, hence so is $\phi^r$ for all $r$. For $f(t)\in \mathbb Z[t]$ let $\alpha_f\in \E$ be given by $\alpha_f=f(\phi)$. Then  
	
	\[\alpha_f\left(\begin{bmatrix}x & 0 \\ 0 & 1\end{bmatrix}\right)=\begin{bmatrix} x^{f(p)} & 0 \\ 0 & 1\end{bmatrix}\]
	and we obtain the group operations
	\begin{align*}\begin{bmatrix}a & b \\ c & d\end{bmatrix}\circ_{\psi_1,\alpha_f} \begin{bmatrix}k & \ell \\ m & n\end{bmatrix}=\begin{bmatrix}ak+bm\delta^{-f(p)} & a\ell\delta^{f(p)}+bn \\ ck+dm\delta^{-f(p)} & c\ell\delta^{f(p)}+dn\end{bmatrix}\\
	\begin{bmatrix}a & b \\ c & d\end{bmatrix}\circ_{\psi_2,\alpha_f} \begin{bmatrix}k & \ell \\ m & n\end{bmatrix}=\begin{bmatrix}ak+bm\delta^{f(p)} & a\ell\delta^{-f(p)}+bn \\ ck+dm\delta^{f(p)} & c\ell\delta^{-f(p)}+dn\end{bmatrix},	
	\end{align*}
	where $\delta=\det A\in\mathbb F_p(\!(x)\!)$.
	We can see that $(G,\circ_{\psi_t,\alpha_{f_1}})=(G,\circ_{\psi_u,\alpha_{f_2}})$ if and only if $f_1(t)\equiv(-1)^{t-u}f_2(t)\pmod {p\mathbb Z[t]}$. Thus the resulting brace block is infinite.
	
Return to the general case $G=\GL_n(F)$.	While it may be tricky to determine the isomorphism class of the groups $(G,\circ):=(G,\circ_{\psi_t,\alpha_f})$ we do know that $\alpha=0$ results in a group isomorphic to $\GL_n(F)$, and $\alpha=-1$ results in a group isomorphic to $F^{\times}\times \SL_n(F)$. The former constructs the trivial brace, and the latter can be seen (shown here in the case $t=1$) by observing that $A\in\GL_n(F)$ can be written as $A=A_0A_1$, where 
	\[A_0=\begin{bmatrix} \delta & 0 \\ 0 & 1 \end{bmatrix},\;\delta=\det(A)\]
	and $A_1\in\SL_n(F)$. Then $\psi_1(A)=\psi_1(A_0A_1)=A_1$ and we get
	\begin{align*}
		A\circ B &= A \psi_1(A^{-1})B\psi_1(A^{-1})^{-1}\\
		&= (A_0A_1) A_1^{-1} B A_1 \\
		&= A_0 B A_1.
	\end{align*}
	The isomorphism $F^{\times}\times \SL_n(F)\to(G,\circ)$ is given by $(\delta,A_1)\mapsto A$, where $A$ is obtained from $A_1$ by multiplying the top row by $\delta$. 
\end{example}

\bibliographystyle{alpha} 
\bibliography{../../MyRefs}
\end{document}